\documentclass[12pt, reqno]{amsart}
\usepackage{amsmath, amsthm, amscd, amsfonts, amssymb, graphicx, color}
\usepackage[bookmarksnumbered, colorlinks, plainpages]{hyperref}
\hypersetup{colorlinks=true,linkcolor=red, anchorcolor=green, citecolor=cyan, urlcolor=red, filecolor=magenta, pdftoolbar=true}

\newtheorem{theorem}{Theorem}[section]
\newtheorem{lemma}[theorem]{Lemma}
\newtheorem{prop}[theorem]{Proposition}
\newtheorem{cor}[theorem]{Corollary}
\theoremstyle{definition}

\newtheorem{example}[theorem]{Example}

\theoremstyle{remark}
\newtheorem{remark}[theorem]{\bf{Remark}}
\numberwithin{equation}{section}
\begin{document}

\title [Numerical radius inequalities]{Proper improvement of well-known numerical radius inequalities and their applications }

\author[P. Bhunia and K. Paul]{Pintu Bhunia and Kallol Paul}

\address{(Bhunia) Department of Mathematics, Jadavpur University, Kolkata 700032, India}
\email{pintubhunia5206@gmail.com}

\address{(Paul) Department of Mathematics, Jadavpur University, Kolkata 700032, India}
\email{kalloldada@gmail.com;kallol.paul@jadavpuruniversity.in}

\thanks{First author  would like to thank UGC, Govt. of India for the financial support in the form of SRF. Prof. Kallol Paul would like to thank RUSA 2.0, Jadavpur University for the partial support.}
\thanks{}
\thanks{}



\maketitle

\begin{abstract}
New inequalities for the numerical radius of bounded linear operators defined on a complex Hilbert space $\mathcal{H}$ are given. In particular, it is established that  if $T$ is a bounded linear operator on a Hilbert space $\mathcal{H}$ then
\[ w^2(T)\leq \min_{0\leq \alpha \leq 1} \left \| \alpha T^*T +(1-\alpha)TT^* \right \|,\] where $w(T)$ is the numerical radius of $T.$  The inequalities obtained here are non-trivial improvement of the well-known numerical radius inequalities. As an application  we estimate bounds  for the zeros of a complex monic polynomial.

\end{abstract}

\section{Introduction}

\noindent Let $\mathcal{B}(\mathcal{H})$ denote the $\mathbb{C}^*$-algebra of all bounded linear operators defined on a complex Hilbert space $\mathcal{H}$ with inner product $\langle.,.\rangle$. For  $T \in \mathcal{B}(\mathcal{H})$, $T^*$ denotes the adjoint of $T$. Also  $|T|, |T^*|$ denote the positive  operators  $(T^*T)^{\frac{1}{2}}, (TT^*)^{\frac{1}{2}}$ respectively. The alphabet `` $O$ '' stands for the zero operator on $\mathcal{H}$. The numerical range of $T$, denoted by $W(T),$ is defined as  \[W(T)=\{ \langle Tx,x\rangle~~:~~x\in \mathcal{H}, ~~\|x\|=1 \}.\]
Two important constants associated with numerical range of $T$ are the numerical radius $w(T)$ of $T$ and the Crawford number $c(T)$ of $T$,  which are defined respectively as 
\begin{eqnarray*}
w(T)=\sup \left \{ ~~|\lambda|~:~\lambda \in W(T)~~\right \} ~~\mbox{and} ~~c(T)=\inf \left \{~~|\lambda|~:~\lambda \in W(T)~~\right \}.
\end{eqnarray*}
For  $T \in \mathcal{B}(\mathcal{H})$, let $\|T\|$ be the operator norm of $T$.  The numerical radius satisfies the following well-known inequality 
\begin{eqnarray}\label{eqv}
\frac{1}{2}\|T\|\leq w(T)\leq \|T\|.
\end{eqnarray}
The spectral radius of $T,$ denoted as $r(T)$, is defined as the radius of the smallest circle with center at origin containing the spectrum $ \sigma(T)$ of the operator $T$.  It is well-known that closure of the numerical range contains the spectrum and so $ r(T) \leq w(T).$ 
The inequality  (\ref{eqv}) is sharp,  $w(T) = \|T\|$ if $T$ is normal and $ w(T)=\frac{1}{2}\|T\| $ if $T^2=O.$
Over the years, various numerical radius inequalities have been obtained to improve on the inequality (\ref{eqv}). Interested readers  can look into  \cite{abu,bag,bhunia1,bhunia2,kittaneh1,kittaneh2} and the references therein for more information on recent advances in numerical radius  inequalities.

In this paper, we establish some new inequalities for the numerical radius of bounded linear operators. In particular we obtain the inequalities 
\[ w^2(T)\leq \min_{0\leq \alpha \leq 1} \left \| \alpha |T|^2 +(1-\alpha)|T^*|^2 \right \|,\] 
\[w^{2}(T)\leq \min_{0\leq \alpha \leq 1}  \left \{\frac{\alpha}{2} w(T^2) +\left \| \frac{\alpha}{4}   |T|^{2}+(1-\frac{3}{4}\alpha)  |T^*|^{2}  \right \| \right \} .\]
We show that the inequalities obtained here greneralize and improve on the existing well-known inequalities given in \cite{abu,kittaneh1,kittaneh2}. As an application of the numerical radius inequalities obtained here we give a better estimation of the bounds for the zeros of a complex monic polynomial.

 \section{Main Results}

 We begin this section with the following proposition that gives an inequality involving the operator norm and the Crawford number of bounded linear operators.

\begin{prop}\label{prop1}
Let $T\in \mathcal{B}(\mathcal{H})$. Then the following inequality holds.
\[\|T\|^2+\max \left \{c(|T|^2), c(|T^*|^2)  \right \}\leq \|T^*T+TT^*\|.\]
\end{prop}

\begin{proof}
The proof follows from the observation that $ \forall x \in \mathcal{H}$ with $  \|x\|=1 $ we have,
 $ \|Tx\|^2+\|T^*x\|^2=\langle (T^*T+TT^*)x,x\rangle \leq \|T^*T+TT^*\|.$ 
\end{proof}

To proceed further we need the following lemmas.

\begin{lemma}$($\cite[pp. 75-76]{halmos}$)$\label{lemma1}
Let $T\in \mathcal{B}(\mathcal{H})$ and let $x\in \mathcal{H}$. Then 
\[|\langle Tx,x\rangle|\leq \langle |T|x,x\rangle^{1/2}~~\langle |T^*|x,x\rangle^{1/2}.\]
\end{lemma}

\begin{lemma}$($\cite[p. 20]{simon}$)$\label{lemma2}
Let $T\in \mathcal{B}(\mathcal{H})$ be positive and let $x\in \mathcal{H}$ with $\|x\|=1$. Then \[\langle Tx,x\rangle^r\leq \langle T^rx,x\rangle, ~~\forall~~r\geq 1.\]
\end{lemma}

Now, we are in a position to present our first theorem.

\begin{theorem}\label{theorem1}
Let $T\in \mathcal{B}(\mathcal{H})$. Then 
\[w^{2r}(T)\leq \left\|\alpha |T|^{2r}+(1-\alpha)|T^*|^{2r}  \right\|,~~\forall~~r\geq 1,~\mbox{and}~~\forall~\alpha,~0\leq \alpha\leq 1.\]

\end{theorem}

\begin{proof}
Let $x\in \mathcal{H}$ with $\|x\|=1$. Then by Cauchy-Schwarz inequality, we get $$|\langle Tx,x\rangle|=\alpha |\langle Tx,x\rangle|+(1-\alpha)|\langle Tx,x\rangle|\leq \alpha |\langle Tx,x\rangle|+(1-\alpha)\|T^*x\|.$$
Therefore, by the convexity of the function $f(t)=t^{2r}$, we get
\begin{eqnarray*}
|\langle Tx,x\rangle|^{2r}&\leq& \alpha |\langle Tx,x\rangle|^{2r}+(1-\alpha)\|T^*x\|^{2r}\\
&\leq &  \alpha \langle |T|x,x\rangle^{r}~~\langle |T^*|x,x\rangle^{r}+(1-\alpha)\langle |T^*|^2x,x\rangle ^{r}, ~~\mbox{by Lemma \ref{lemma1}}\\
&\leq & \alpha \langle |T|^{r}x,x\rangle~~\langle |T^*|^{r}x,x\rangle+(1-\alpha)\langle |T^*|^{2r}x,x\rangle, ~~\mbox{by Lemma \ref{lemma2}}\\
&\leq & \frac{\alpha}{2} \left (\langle |T|^{r}x,x\rangle^2 + \langle |T^*|^{r}x,x\rangle^2 \right ) +(1-\alpha)\langle |T^*|^{2r}x,x\rangle,\\
&& \,\,\,\,\,\,\,\,\,\,\,\,\,\,\,\,\,\,\,\,\,\,\,\,\,\,\,\,\,\,\,\,\,\,\,\,\,\,\,\,\,\,\,\,\,\,\,\,\,\,\,\,\,\,\,\,\,\,\,\,\,\,\,\,\,\,\,\,\,\,\,\,\,\,\,\,\,\,\,\,\,\,\,\,\,\,\,\,\,\,\,\,\,\,\,\,~~\mbox{by AM-GM inequality}\\
&\leq & \frac{\alpha}{2} \left (\langle |T|^{2r}x,x \rangle + \langle |T^*|^{2r}x,x \rangle \right ) +(1-\alpha) \langle |T^*|^{2r}x,x \rangle, ~~\mbox{by Lemma \ref{lemma2}}\\
&=& \left \langle \left \{ \frac{\alpha}{2} \left ( |T|^{2r}+ |T^*|^{2r}\right )+(1-\alpha) |T^*|^{2r} \right\} x,x \right \rangle\\ 
&\leq & \left \| \frac{\alpha}{2} \left ( |T|^{2r}+ |T^*|^{2r}\right )+(1-\alpha) |T^*|^{2r} \right \|\\
&=& \left\|  \frac{\alpha}{2} |T|^{2r}+\left(1-\frac{\alpha}{2}\right)|T^*|^{2r}  \right\|. 
\end{eqnarray*}
Taking supremum over all $x\in \mathcal{H}$ with $\|x\|=1$, we get 
\[w^{2r}(T)\leq \left\| \frac{\alpha}{2} |T|^{2r}+(1-\frac{\alpha}{2}) |T^*|^{2r}  \right\|,~~\forall~~r\geq 1,~~\forall~\alpha,~0\leq \alpha\leq 1.\]
Replacing $T$ by $T^*$ in the above inequality we get,
\[w^{2r}(T)\leq \left\| \left(1-\frac{\alpha}{2}\right)  |T|^{2r}+\frac{\alpha}{2} |T^*|^{2r}  \right\|,~~\forall~~r\geq 1,~~\forall~\alpha,~0\leq \alpha\leq 1.\]
Combining the above two inequalities we get the desired inequality.
\end{proof}
As a consequence of  Theorem \ref{theorem1} we easily get the following corollary.

\begin{cor}\label{inequality1}
Let $T\in \mathcal{B}(\mathcal{H})$. Then 	
\begin{eqnarray}\label{bound1comb}
w^2(T)& \leq &  \min_{0\leq \alpha \leq 1} \left \| \alpha |T|^2 +(1-\alpha)|T^*|^2 \right \|.
\end{eqnarray} 
\end{cor}

In \cite{kittaneh1}, Kittaneh proved that the following inequality
\begin{eqnarray}\label{kittaneh1}
w^{2}(T)\leq  \frac{1}{2}\left\| |T|^{2}+|T^*|^{2}  \right\|.
\end{eqnarray}

Inequalities obtained in Theorem \ref{theorem1} and Corollary \ref{inequality1} generalize and improve on the inequality (\ref{kittaneh1}) obtained by Kittaneh. In order to appreciate our inequality (\ref{bound1comb}), we give the following examples which show that $$\min_{0\leq \alpha \leq 1} \left\| \alpha |T|^{2}+\left(1-\alpha\right)|T^*|^{2}  \right\| < \frac{1}{2}\left\| |T|^{2}+|T^*|^{2}  \right\| $$ and imply that our inequality (\ref{bound1comb}) is a non-trivial improvement of the inequality (\ref{kittaneh1}).

\begin{example}\label{example1}
(i) Let  $$T =  \left(\begin{array}{ccc}
	0&1&0 \\
	0&0&2\\
	0&0&0
	\end{array}\right).$$ Then $$|T|^2 =  \left(\begin{array}{ccc}
	0&0&0 \\
	0&1&0\\
	0&0&4
	\end{array}\right)~~ \mbox{and}~~ |T^*|^2 =  \left(\begin{array}{ccc}
	1&0&0 \\
	0&4&0\\
	0&0&0
	\end{array}\right).$$ Therefore,
	\begin{eqnarray*}
	 \min_{0\leq \alpha \leq 1} \left\| \alpha |T|^{2}+\left(1-\alpha\right)|T^*|^{2}  \right\|=\min_{0\leq \alpha \leq 1}\max\{1-\alpha,4-3\alpha,4\alpha\} =\frac{16}{7}  	 
	\end{eqnarray*}
and   \[\frac{1}{2}\left\| |T|^{2}+|T^*|^{2}  \right\|	 =  \frac{5}{2}.\] 
Thus,  \[\min_{0\leq \alpha \leq 1} \left\| \alpha |T|^{2}+\left(1-\alpha\right)|T^*|^{2}  \right\| < \frac{1}{2}\left\| |T|^{2}+|T^*|^{2}  \right\|.\]
(ii) Let  $$S =  \left(\begin{array}{cccc}
	0&2&0&0 \\
	0&0&3&0\\
	0&0&0&0\\
	0&0&0&1
	\end{array}\right).$$ Then $$|S|^2 =  \left(\begin{array}{cccc}
	0&0&0&0 \\
	0&4&0&0\\
	0&0&9&0\\
	0&0&0&1
	\end{array}\right)~~ \mbox{and}~~ |S^*|^2 =  \left(\begin{array}{cccc}
	4&0&0&0 \\
	0&9&0&0\\
	0&0&0&0\\
	0&0&0&1
	\end{array}\right).$$ Therefore,
	\begin{eqnarray*}
	\min_{0\leq \alpha \leq 1} \left\| \alpha |S|^{2}+\left(1-\alpha\right)|S^*|^{2}  \right\|&=&\frac{81}{14}<\frac{13}{2}=\frac{1}{2}\left\| |S|^{2}+|S^*|^{2}  \right\|. 
\end{eqnarray*}

\end{example}

To prove our next theorem we need the well-known Heinz inequality.

\begin{theorem}[Heinz inequality \cite{kato}]\label{Heinz}
Let $T\in \mathcal{B}(\mathcal{H}).$  Then for all  $x,y\in \mathcal{H}$, 
\begin{eqnarray}
|\langle Tx,y\rangle|^2\leq \langle |T|^{2\lambda}x,x\rangle~~\langle |T^*|^{2(1-\lambda)}y,y\rangle, ~\forall~~ \lambda,~0\leq \lambda \leq 1.
\end{eqnarray}
\end{theorem}
We note that Lemma \ref{lemma1} is a special case of the Heinz inequality.  Now proceeding in the same way as Theorem \ref{theorem1} we get the following theorem.
\begin{theorem}\label{heinz}
Let $T\in \mathcal{B}(\mathcal{H})$. Then
\begin{eqnarray}
w^{2r}(T)\leq \left \|\frac{\alpha}{2}\left(|T|^{4\lambda r}+|T^*|^{4(1-\lambda)r}  \right)+ (1-\alpha) |T^*|^{2r}  \right\|
\end{eqnarray}
and 
\begin{eqnarray}
w^{2r}(T)\leq \left \|\frac{\alpha}{2}\left(|T|^{4\lambda r}+|T^*|^{4(1-\lambda)r}  \right)+ (1-\alpha) |T|^{2r}  \right\|,
\end{eqnarray}
$ \forall ~~r\geq 1$ and $\forall ~\alpha,\lambda $ with $0\leq \alpha,\lambda \leq 1$.
\end{theorem}

In \cite{abu}, Abu-Omar and Kittaneh proved that the following inequality
\begin{eqnarray}\label{kittaneh2}
w^{2}(T)\leq  \frac{1}{2} w(T^2) +\frac{1}{4}  \left\|   |T|^{2}+  |T^*|^{2}  \right\|.
\end{eqnarray}
In our next theorem we generalize and improve on the inequality   (\ref{kittaneh2}). To do so we need the following inequality.

\begin{lemma}(Buzano \cite{buzano})\label{lemma3}
Let $a,e,b\in \mathcal{H}$ with $\|e\|=1.$ Then 
\[|\langle a,e\rangle~~\langle e,b\rangle|\leq \frac{1}{2}\left(\|a\|~~\|b\|+|\langle a,b\rangle|\right).\]
\end{lemma}

Using  Buzano's inequality we first prove the following lemma.

\begin{lemma}
	Let $T\in \mathcal{B}(\mathcal{H})$ and let $x\in \mathcal{H}$ with $\|x\|=1$. Then 
	\begin{eqnarray}
	|\langle Tx,x\rangle|^{2r} \leq \frac{1}{2} |\langle T^2x,x\rangle|^{r}+\frac{1}{4} \left \langle (|T|^{2r}+|T^*|^{2r})x,x \right \rangle,~~\forall~~ r\geq 1.
	\end{eqnarray}
\end{lemma}

\begin{proof}
	Taking $a=Tx$, $b=T^*x$ and $e=x$ in Lemma \ref{lemma3}, we get
	\[|\langle Tx,x\rangle|^2\leq \frac{1}{2}\left(|\langle T^2x,x\rangle|+\|Tx\|~~\|T^*x\|\right).\]
	By convexity of the function $f(t)=t^r~(r \geq 1)$, we get
	\begin{eqnarray*}
		|\langle Tx,x\rangle|^{2r} &\leq & \frac{1}{2}\left(|\langle T^2x,x\rangle|^r+\|Tx\|^r~~\|T^*x\|^r\right)\\
		&\leq & \frac{1}{2}\left(|\langle T^2x,x\rangle|^r+\frac{1}{2}(\|Tx\|^{2r}+\|T^*x\|^{2r})\right),~~\mbox{by AM-GM inequality}\\
		&= & \frac{1}{2}\left(|\langle T^2x,x\rangle|^r+\frac{1}{2}(\langle |T|^2x,x\rangle ^{r}+\langle |T^*|^2x,x\rangle ^{r})\right)\\
		&\leq & \frac{1}{2}\left(|\langle T^2x,x\rangle|^r+\frac{1}{2}(\langle |T| ^{2r}x,x\rangle+\langle |T^*|^{2r}x,x\rangle)\right), ~~\mbox{by Lemma \ref{lemma2}}\\
		&=& \frac{1}{2} |\langle T^2x,x\rangle|^{r}+\frac{1}{4} \left \langle (|T|^{2r}+|T^*|^{2r})x,x \right \rangle.
	\end{eqnarray*}
	This completes the proof.
\end{proof}
Now, we present the desired theorem.

\begin{theorem}\label{theorem2}
Let $T\in \mathcal{B}(\mathcal{H})$. Then  $ \forall~~r\geq 1 $ and $ \forall~\alpha,~0\leq \alpha\leq 1,$ 
\begin{eqnarray}\label{inequality2}
(i) ~w^{2r}(T) &\leq & \frac{\alpha}{2} w^r(T^2) +\left\| \frac{\alpha}{4}   |T|^{2r}+\left(1-\frac{3}{4}\alpha\right)  |T^*|^{2r}  \right\|,
\end{eqnarray}
\begin{eqnarray}\label{inequality3}
(ii)~w^{2r}(T) &\leq & \frac{\alpha}{2} w^r(T^2) +\left\| (1-\frac{3}{4} \alpha)  |T|^{2r}+\frac{\alpha}{4} |T^*|^{2r}  \right\|.
\end{eqnarray}
\end{theorem}
\begin{proof}
 Let $x\in \mathcal{H}$ with $\|x\|=1$. Then by Cauchy-Schwarz inequality, we get $$|\langle Tx,x\rangle|=\alpha |\langle Tx,x\rangle|+(1-\alpha)|\langle Tx,x\rangle|\leq \alpha |\langle Tx,x\rangle|+(1-\alpha)\|T^*x\|, ~\forall~~ \alpha, 0 \leq \alpha \leq1.$$
By convexity of the function $f(t)=t^{2r}~(r \geq1)$, we get
\begin{eqnarray*}
|\langle Tx,x\rangle|^{2r}&\leq& \alpha |\langle Tx,x\rangle|^{2r}+(1-\alpha)\|T^*x\|^{2r}\\ 
&\leq& \alpha |\langle Tx,x\rangle|^{2r}+(1-\alpha)\langle |T^*|^{2r}x,x\rangle, ~~\mbox{by Lemma \ref{lemma2}}\\
&\leq& \frac{\alpha }{2} |\langle T^2x,x\rangle|^{r}+\frac{\alpha }{4} \left \langle (|T|^{2r}+|T^*|^{2r})x,x \right \rangle+(1-\alpha)\langle |T^*|^{2r}x,x\rangle, \\
&&\,\,\,\,\,\,\,\,\,\,\,\,\,\,\,\,\,\,\,\,\,\,\,\,\,\,\,\,\,\,\,\,\,\,\,\,\,\,\,\,\,\,\,\,\,\,\,\,\,\,\,\,\,\,\,\,\,\,\,\,\,\,\,\,\,\,\,\,\,\,\,\,\,\,\,\,\,\,\,\,\,\,\,\,\,\,\,\,\,\,\,\,\,\,\,\,~~\mbox{by Lemma \ref{lemma3}}\\
&=& \frac{\alpha }{2} |\langle T^2x,x\rangle|^{r}+ \left \langle \left \{ \frac{\alpha}{4} \left ( |T|^{2r}+ |T^*|^{2r}\right )+(1-\alpha) |T^*|^{2r} \right\} x,x \right \rangle\\
&=& \frac{\alpha }{2} |\langle T^2x,x\rangle|^{r}+\left \langle \left \{  \frac{\alpha}{4}   |T|^{2r}+\left(1-\frac{3}{4}\alpha\right)  |T^*|^{2r}     \right\} x,x \right \rangle\\
&\leq& \frac{\alpha}{2} w^r(T^2) +\left\| \frac{\alpha}{4}   |T|^{2r}+\left(1-\frac{3}{4}\alpha\right)  |T^*|^{2r}  \right\|.
\end{eqnarray*}
Taking supremum over all $x\in \mathcal{H}$ with $\|x\|=1$, we get the inequality (\ref{inequality2}). 
Replacing $T$ by $T^*$ in the inequality (\ref{inequality2}) we get the inequality (\ref{inequality3}). This completes the proof.
\end{proof}
As a consequence we get the following upper bound for the numerical radius.
\begin{cor} \label{cor2}
	Let $T \in \mathcal{B}(\mathcal{H}).$ Then 
\begin{eqnarray}\label{bound2comb}
w^{2}(T)\leq \min \left \{\beta_1, \beta_2  \right \}, 
\end{eqnarray}
\mbox{where}
\[ \beta_1 = \min_{0\leq \alpha \leq 1} \left \{\frac{\alpha}{2} w(T^2) +\left\| \frac{\alpha}{4}   |T|^{2}+\left(1-\frac{3}{4}\alpha\right)  |T^*|^{2}  \right\|\right \}\] and
 \[\beta_2 = \min_{0\leq \alpha \leq 1} \left \{\frac{\alpha}{2} w(T^2) +\left\|  \left(1-\frac{3}{4}\alpha\right)   |T|^{2}+\frac{\alpha}{4}  |T^*|^{2}  \right\|\right \}.\]
\end{cor}
\begin{proof}
The proof follows easily by taking $r=1$ in the inequalities (\ref{inequality2}) and (\ref{inequality3}).
\end{proof}

Inequalities obtained in Theorem \ref{theorem2} and Corollary \ref{cor2} generalize and improve on the inequality (\ref{kittaneh2}) obtained by Abu-Omar and Kittaneh. In order to appreciate our inequality (\ref{bound2comb}), we give the following examples, it shows that  the inequality (\ref{bound2comb}) is a non-trivial improvement of the inequality (\ref{kittaneh2}).

\begin{example}
(i) Let  $$T =  \left(\begin{array}{ccc}
	0&1&0 \\
	0&0&2\\
	0&0&0
	\end{array}\right).$$ Then by elementary calculations, we get $\beta_1=\frac{7}{4}$ and $\beta_2=\frac{22}{13}$. Therefore, 
	\[\min \left \{\beta_1, \beta_2  \right \}=\frac{22}{13}<\frac{7}{4}=\frac{1}{2} w(T^2) +\frac{1}{4}  \left\|   |T|^{2}+  |T^*|^{2}  \right\|.\]
(ii) Let  $$S =  \left(\begin{array}{cccc}
	0&2&0&0 \\
	0&0&3&0\\
	0&0&0&0\\
	0&0&0&1
	\end{array}\right).$$ Then by elementary calculations, we get $\beta_1=\frac{19}{4}$ and $\beta_2=\frac{37}{8}$. Therefore, 
	\[\min \left \{\beta_1, \beta_2  \right \}=\frac{37}{8}<\frac{19}{4}=\frac{1}{2} w(S^2) +\frac{1}{4}  \left\|   |S|^{2}+  |S^*|^{2}  \right\|.\]

\end{example}
We next prove the following theorem.
\begin{theorem}\label{theorem3}
Let $T\in \mathcal{B}(\mathcal{H})$. Then  $ \forall~~r\geq 1 $ and $ \forall~\alpha,~0\leq \alpha\leq 1 $ we have
\begin{eqnarray}\label{inequality2.11}
w^{2r}(T) & \leq & \left\| \alpha   \left(\frac{|T|+|T^*|}{2}\right)^{2r}+\left(1-\alpha\right)  |T^*|^{2r}  \right\|.
\end{eqnarray}
\begin{eqnarray}\label{inequality2.12}
w^{2r}(T) & \leq & \left\| \alpha   \left(\frac{|T|+|T^*|}{2}\right)^{2r}+\left(1-\alpha\right)  |T|^{2r}  \right\|.
\end{eqnarray}
\end{theorem}

\begin{proof}
Let $x\in \mathcal{H}$ with $\|x\|=1$. Then by Cauchy-Schwarz inequality, we get $ \forall~~ \alpha,~~ 0 \leq \alpha \leq 1,$
$$|\langle Tx,x\rangle|=\alpha |\langle Tx,x\rangle|+(1-\alpha)|\langle Tx,x\rangle|\leq \alpha |\langle Tx,x\rangle|+(1-\alpha)\|T^*x\|.$$
By the convexity of the function $f(t)=t^{2r}~(r \geq 1) $, we get
\begin{eqnarray*}
|\langle Tx,x\rangle|^{2r}&\leq& \alpha |\langle Tx,x\rangle|^{2r}+(1-\alpha)\|T^*x\|^{2r}\\ 
&\leq& \alpha |\langle Tx,x\rangle|^{2r}+(1-\alpha)\langle |T^*|^{2r}x,x\rangle, ~~\mbox{by Lemma \ref{lemma2}}\\
&\leq& \alpha \left(\langle |T|x,x\rangle^{1/2}\langle |T^*|x,x\rangle^{1/2}\right)^{2r}+(1-\alpha)\langle |T^*|^{2r}x,x\rangle, ~~\mbox{by Lemma \ref{lemma1}}\\
&\leq& \alpha \left(\frac{\langle |T|x,x\rangle+\langle |T^*|x,x\rangle}{2} \right)^{2r}+(1-\alpha)\langle |T^*|^{2r}x,x\rangle, \\
&& \,\,\,\,\,\,\,\,\,\,\,\,\,\,\,\,\,\,\,\,\,\,\,\,\,\,\,\,\,\,\,\,\,\,\,\,\,\,\,\,\,\,\,\,\,\,\,\,\,\,\,\,\,\,\,\,\,\,\,\,\,\,\,\,\,\,\,\,\,\,\,\,\,\,\,\,\,\,\,\,\,\,\,\,\,\,\,\,\,\,\,\,\,\,\,\,~~\mbox{by AM-GM inequality}\\
&=& \alpha \left(\frac{\langle (|T|+ |T^*|)x,x\rangle}{2} \right)^{2r}+(1-\alpha)\langle |T^*|^{2r}x,x\rangle \\
&\leq& \alpha \left \langle \left(\frac{|T|+ |T^*|}{2}\right)^{2r}x,x \right \rangle +(1-\alpha)\langle |T^*|^{2r}x,x\rangle, ~~\mbox{by Lemma \ref{lemma2}}\\
&=& \left \langle \left\{   \alpha   \left(\frac{|T|+|T^*|}{2}\right)^{2r}+\left(1-\alpha\right)  |T^*|^{2r}   \right\} \right \rangle\\
& \leq & \left\| \alpha   \left(\frac{|T|+|T^*|}{2}\right)^{2r}+\left(1-\alpha\right)  |T^*|^{2r}  \right\|.
\end{eqnarray*}
Taking supremum over all $x\in \mathcal{H}$ with $\|x\|=1$, we get the inequality (\ref{inequality2.11}).
Replacing $T$ by $T^*$ in the inequality (\ref{inequality2.11}), we get that the inequality (\ref{inequality2.12}).
\end{proof}
The following corollary is an easy conequence of Theorem \ref{theorem3}.
\begin{cor}\label{cor3}
Let $T\in \mathcal{B}(\mathcal{H})$. Then  
\begin{eqnarray}\label{bound3comb}
w^{2}(T)&\leq &\min \{ \gamma_1, \gamma_2\},
\end{eqnarray}
 \mbox{where }
$$ \gamma_1 = \min_{0\leq \alpha \leq 1} \left\| \alpha   \left(\frac{|T|+|T^*|}{2}\right)^{2}+\left(1-\alpha\right)  |T^*|^{2}  \right\| $$ 
and
$$ \gamma_2 = \min_{0\leq \alpha \leq 1} \left\| \alpha   \left(\frac{|T|+|T^*|}{2}\right)^{2}+\left(1-\alpha\right)  |T|^{2}  \right\|.$$
\end{cor}

\begin{remark}
In \cite{kittaneh2}, Kittaneh proved that the following inequality
\begin{eqnarray}\label{kittaneh3}
w^{2}(T)\leq \frac{1}{4}\left\| ~~|T|+|T^*|  ~~\right\|^2.
\end{eqnarray}
The inequalities obtained in Theorem \ref{theorem3} and Corollary \ref{cor3} generalize and improve on the inequality (\ref{kittaneh3}) obtained by Kittaneh. As before considering the operators $T$ and $S$ used in Example \ref{example1} we can show  that the inequality (\ref{bound3comb}) is a non-trivial improvement of the inequality (\ref{kittaneh3}).
\end{remark}

\section{\textbf{ Application to estimate the modulus of zeros of polynomials }}

 \noindent Let $p(z)=z^n+a_{n-1}z^{n-1}+a_{n-2}z^{n-2}+\ldots+a_1z+a_0$ be a monic polynomial of degree $n(\geq 2)$ with complex coefficients $a_i$, $i=0,1,\ldots,n-1$. Then the Frobenius companion matrix of $p(z)$ is 
$$C(p)=\left(\begin{array}{ccccc}
-a_{n-1}&-a_{n-2}&\ldots&-a_1&-a_0\\
1&0&\ldots&0&0\\
0&1&\ldots&0&0\\
\vdots&\vdots& &\vdots&\vdots\\
0&0&\ldots&1&0\\
\end{array}\right)_{n\times n}.$$
It is well-known that all the eigenvalues of $C(p)$ are exactly the zeros of the polynomial $p(z).$ We consider $C(p)$ as a bounded linear operator on $\mathbb{C}^n$. Therefore, it follows from the inequality $r(C(p))\leq w(C(p))$ that if $\lambda$ is a zero of $p(z)$ then $|\lambda|\leq w(C(p)).$ Using this argument, we obtain estimation for the bounds of the zeros of the polynomial $p(z)$. In order to achive our goal, we first prove the following inequality for the numerical radius of $2\times 2$ block matrices.

\begin{theorem}\label{theorem4}
Let $\mathcal{H}_1$, $\mathcal{H}_2$ be two complex Hilbert spaces and let $\mathbb{T}=\left(\begin{array}{cc}
	O&B \\
	C&O
	\end{array}\right)\in \mathcal{B}(\mathcal{H}_1 \oplus \mathcal{H}_2).$ Then 
	\[ w^2(\mathbb{T})\leq  \min_{0\leq \alpha\leq 1} \max \left\{\|(1-\alpha)BB^*+\alpha C^*C\|,\|\alpha B^*B+ (1-\alpha)CC^*\| \right\}.\]
\end{theorem}

\begin{proof}
By short calculations, we get \[\alpha \mathbb{T}^*\mathbb{T}+(1-\alpha)\mathbb{T}\mathbb{T}^*= \left(\begin{array}{cc}
	(1-\alpha)BB^*+\alpha C^*C &O \\
	O&\alpha B^*B+ (1-\alpha)CC^*
	\end{array}\right).\]  It follows from Corollary \ref{inequality1}  that 
	\begin{eqnarray*}
	 w^2(\mathbb{T}) &\leq& \min_{0\leq \alpha\leq 1} \left \| \left(\begin{array}{cc}
	(1-\alpha)BB^*+\alpha C^*C &O \\
	O&\alpha B^*B+ (1-\alpha)CC^*
	\end{array}\right)\right \|\\
	&=& \min_{0\leq \alpha\leq 1} \max \left\{\|(1-\alpha)BB^*+\alpha C^*C\|,\|\alpha B^*B+ (1-\alpha)CC^*\| \right\}.
	\end{eqnarray*}
This completes the proof of the theorem.
\end{proof}

Next we need the following  two lemmas.

\begin {lemma}$($\cite[Lemma 2.4]{bhunia3}$)$ \label{lemma4}
Let $\mathbb{S}_{n}=\left(\begin{array}{ccccc}
0&0&\ldots&0&0\\
1&0&\ldots&0&0\\
0&1&\ldots&0&0\\
\vdots&\vdots& &\vdots&\vdots\\
0&0&\ldots&1&0\\
\end{array}\right)_{n\times n}$. Then the numerical radius of $\mathbb{S}_n$ is given by  $$w(\mathbb{S}_n)=\cos \left(\frac{\pi}{n+1}\right).$$
\end{lemma}

\begin{lemma}$($\cite[Cor. 2]{abu2}$)$\label{lemma5}
Let $\mathcal{H}_1$, $\mathcal{H}_2$ be two complex Hilbert spaces and let $\mathbb{S}=\left(\begin{array}{cc}
	A&B \\
	C&D
	\end{array}\right)\in \mathcal{B}(\mathcal{H}_1 \oplus \mathcal{H}_2).$ Then
	\[w(\mathbb{S})\leq \frac{1}{2}\left(w(A)+w(D)\right)+\frac{1}{2}\sqrt{\left(w(A)-w(D)\right)^2+4w^2(\mathbb{T})},\]
	where $\mathbb{T}=\left(\begin{array}{cc}
	O&B \\
	C&O
	\end{array}\right).$
\end{lemma}

Now, we are in a position to present our desired estimation for the bounds of the zeros of $p(z)$.

\begin{theorem}\label{theorem5}
Let $\lambda$ be any zero of $p(z)$. Then 
\begin{eqnarray*}
|\lambda| &\leq& \frac{1}{2}\left(|a_{n-1}|+\cos \left(\frac{\pi}{n}\right) \right)+\frac{1}{2}\sqrt{\left(|a_{n-1}|-\cos \left(\frac{\pi}{n}\right)\right)^2+2\left (1 +\sum_{i=0}^{n-2}|a_i|^2\right)}. 
\end{eqnarray*}
\end{theorem}

\begin{proof}
Let  $C(p)=\left(\begin{array}{cc}
	A&B \\
	C&D
\end{array}\right)$, where $ A = (a_{n-1})_{1\times 1}$, $C=\left(\begin{array}{c}
1 \\
0\\
\vdots\\
0
\end{array}\right)_{n-1 \times 1} $,\\  $B=(-a_{n-2} -a_{n-3} \ldots -a_1 -a_0 )_{1\times n-1}$  and $D=\mathbb{S}_{n-1}.$ 
Then, by using Theorem \ref{theorem4}, we get that
\begin{eqnarray*}
w\left(\begin{array}{cc}
	O&B \\
	C&O
\end{array}\right) &\leq&  \min_{0\leq \alpha\leq 1} \max \left\{\sqrt{ \alpha +(1-\alpha)\sum_{i=0}^{n-2}|a_i|^2},\sqrt{ (1-\alpha) +\alpha \sum_{i=0}^{n-2}|a_i|^2} \right\}. 
\end{eqnarray*}
Clearly,  $$\min_{0\leq \alpha\leq 1} \max \left\{{ \alpha +(1-\alpha)\sum_{i=0}^{n-2}|a_i|^2},{ (1-\alpha) +\alpha \sum_{i=0}^{n-2}|a_i|^2} \right\}={ \frac{1}{2}\left (1 +\sum_{i=0}^{n-2}|a_i|^2\right)}.$$
Therefore, from Lemma \ref{lemma5}, we get
\[w(C(p))\leq \frac{1}{2}\left(w(A)+w(D)\right)+\frac{1}{2}\sqrt{\left(w(A)-w(D)\right)^2+2\left (1 +\sum_{i=0}^{n-2}|a_i|^2\right)}.\]
Clearly, $w(A)=|a_{n-1}|$ and $w(D)=\cos \left(\frac{\pi}{n}\right)$, by Lemma \ref{lemma4}. Hence, we get the required inequality of the theorem.

\end{proof}

\begin{example}
We consider a monic polynomial $p(z)=z^5+{2}z^4+{\rm i}z^2-{\rm i}.$ Then $\sum_{i=0}^{3}|a_i|^2=2.$ Therefore, it follows from Theorem \ref{theorem5} that if $\lambda$ be any zero of $p(z)$ then $$|\lambda|\leq \frac{1}{2}\left [2+\cos\left(\frac{\pi}{5}\right)+\sqrt{\left(2-\cos\left(\frac{\pi}{5}\right)\right)^2+6}\right]\approx 2.76634921105.$$

Following  \cite[Section 3]{bhunia3}, we list the bounds for the zeros of the polynomial $p(z)$  to conclude that the bound obtained in Theorem \ref{theorem5} is better than the existing bounds.\\
\begin{center}
\begin{tabular}{ll}
	Cauchy & 3\\
	Montel  & 4 \\
	Alpin et.al. & 3\\
	Fujii and Kubo & 3.090\\
	Paul and Bag & 2.810\\
	Al-Dolat et. al. & 3.325\\
	Abu-Omar and Kittaneh & 2.914\\
	Theorem \ref{theorem5}  & \textbf{2.766}
\end{tabular}
\end{center}
\end{example}

\bibliographystyle{amsplain}

\end{document}